\documentclass[12pt, a4paper,reqno]{amsart}
\usepackage{etex}
\usepackage{pstricks-add}
\usepackage{latexsym}
\usepackage{amssymb}
\usepackage{mathrsfs}
\usepackage{amsmath}
\usepackage{fancybox,color}
\usepackage{enumerate}
\usepackage[latin1]{inputenc}
\usepackage{bbm}
\usepackage{eurosym}
\usepackage{color}
\usepackage{url}
\usepackage{pgf,tikz}\usepackage{mathrsfs}\usetikzlibrary{arrows}

\newcommand{\kom}[1]{}
\renewcommand{\kom}[1]{{\bf [#1]}}

\renewcommand{\theequation}{\arabic{section}.\arabic{equation}}

\usepackage{hyperref}
\hypersetup{
    colorlinks=true,                          
    linkcolor=blue, 
    citecolor=red, 
    urlcolor=blue  } 
 \def\1{\raisebox{2pt}{\rm{$\chi$}}}

\newtheorem{theorem}{Theorem}[section]

\newtheorem{lemma}[theorem]{Lemma}
\newtheorem{proposition}[theorem]{Proposition}
\newtheorem{remark}[theorem]{Remark}

\newcommand{\R}{{\mathbb R}}

\newcommand{\N}{{\mathbb N}}

\newcommand{\E}{{\mathbb E\,}}
\newcommand{\I}{{\text{I}}}
\newcommand{\J}{{\text{II}}}

 \newcommand{\eps}{{\varepsilon}}
 \def\1{\raisebox{2pt}{\rm{$\chi$}}}
 
\newcommand{\abs}[1]{\left|#1\right|}
\newcommand{\norm}[1]{\left|\left|#1\right|\right|}
\newcommand{\Rn}{\mathbb{R}^n}

\def\vint_#1{\mathchoice%
          {\mathop{\kern 0.2em\vrule width 0.6em height 0.69678ex depth -0.58065ex
                  \kern -0.8em \intop}\nolimits_{\kern -0.4em#1}}%
          {\mathop{\kern 0.1em\vrule width 0.5em height 0.69678ex depth -0.60387ex
                  \kern -0.6em \intop}\nolimits_{#1}}%
          {\mathop{\kern 0.1em\vrule width 0.5em height 0.69678ex
              depth -0.60387ex
                  \kern -0.6em \intop}\nolimits_{#1}}%
          {\mathop{\kern 0.1em\vrule width 0.5em height 0.69678ex depth -0.60387ex
                  \kern -0.6em \intop}\nolimits_{#1}}}
\def\vintslides_#1{\mathchoice%
          {\mathop{\kern 0.1em\vrule width 0.5em height 0.697ex depth -0.581ex
                  \kern -0.6em \intop}\nolimits_{\kern -0.4em#1}}%
          {\mathop{\kern 0.1em\vrule width 0.3em height 0.697ex depth -0.604ex
                  \kern -0.4em \intop}\nolimits_{#1}}%
          {\mathop{\kern 0.1em\vrule width 0.3em height 0.697ex depth -0.604ex
                  \kern -0.4em \intop}\nolimits_{#1}}%
          {\mathop{\kern 0.1em\vrule width 0.3em height 0.697ex depth -0.604ex
                  \kern -0.4em \intop}\nolimits_{#1}}}

\newcommand{\aveint}[2]{\mathchoice%
          {\mathop{\kern 0.2em\vrule width 0.6em height 0.69678ex depth -0.58065ex
                  \kern -0.8em \intop}\nolimits_{\kern -0.45em#1}^{#2}}%
          {\mathop{\kern 0.1em\vrule width 0.5em height 0.69678ex depth -0.60387ex
                  \kern -0.6em \intop}\nolimits_{#1}^{#2}}%
          {\mathop{\kern 0.1em\vrule width 0.5em height 0.69678ex depth -0.60387ex
                  \kern -0.6em \intop}\nolimits_{#1}^{#2}}%
          {\mathop{\kern 0.1em\vrule width 0.5em height 0.69678ex depth -0.60387ex
                  \kern -0.6em \intop}\nolimits_{#1}^{#2}}}
\newcommand{\ud}{\, d}

\newcommand{\ol}{\overline}
\newcommand{\Om}{\Omega}
\newcommand{\dist}{\operatorname{dist}}
\renewcommand{\P}{\mathbb{P\,}}
\newcommand{\F}{\mathcal{F}}

\usepackage{pstricks-add}

\usepackage[hmargin=3cm,vmargin=4cm,marginparwidth=2.2cm]{geometry}
\begin{document}

\title[]{Gradient and Lipschitz estimates for tug-of-war type games}

\author[Attouchi]{Amal Attouchi}
\author[Luiro]{Hannes Luiro}
\author[Parviainen]{Mikko Parviainen}
\address{Department of Mathematics and Statistics, University of
Jyv\"askyl\"a, PO~Box~35, FI-40014 Jyv\"askyl\"a, Finland}
\email{amal.a.attouchi@jyu.fi}
\email{hannes.luiro@gmail.com}
\email{mikko.j.parviainen@jyu.fi}


\date{\today}

\begin{abstract}
We define a random step size tug-of-war game, and show that the gradient of a value function exists almost everywhere. We also prove that the gradients of value functions are uniformly bounded and converge weakly to the gradient of the corresponding $p$-harmonic function. Moreover, we establish an improved Lipschitz estimate when boundary values are close to a plane. Such estimates are known to play a key role in higher regularity theory of partial differential equations.  
 The proofs are based on cancellation and coupling methods as well as improved version of the cylinder walk argument.
\end{abstract}

\maketitle
\section{Introduction}
\label{sec:intro}

Higher regularity of value functions to the tug-of-war type games is largely open. In this paper, we develop several techniques in order to study gradient regularity of value functions. In particular, we introduce a version of a tug-of-war  with noise that has, unlike the standard tug-of-war type game, a bounded gradient. We also derive an improved Lipschitz estimate in a ball with boundary values close to a plane. Such estimates are known to play a key role in higher regularity theory of partial differential equations.

The theory of tug-of-war type games has obtained attention after the seminal paper of Peres, Schramm, Sheffield and Wilson \cite{peresssw09} showing that the solutions of the infinity Laplace equation can be approximated by value functions of a two player random turn zero-sum game called tug-of-war. For the $1$-Laplacian  Kohn and Serfaty  established a deterministic game counterpart in \cite{kohns06}. Later Peres and Sheffield introduced a game theoretic approach to the $p$-Laplacian, $1<p<\infty$ \cite{peress08} by using a tug-of-war with noise. The connection between the tug-of-war with noise and $p$-harmonic functions can be compared to the classical connection between the Brownian motion and the Laplace equation.
The $p$-Laplace operator obtained as a limit case also appears in many applications in physics with different values of $p$: electrostatistics and electric networks, non-Newtonian fluids, reaction-diffusion problems, nonlinear elasticity, glaceology, and the thermal radiation of a hydrogen bomb, just to mention a few examples. The analytic and probabilistic results we obtain also apply to this limiting case.
 Moreover, let us also point out that the game theoretic $p$-Laplacian  has gained interest in image processing and machine learning \cite{does11,ett15, elmoataz17, elmo16}. It has also applications in economic modeling \cite{np17}.

In  \cite{manfredipr12} Manfredi, Parviainen and Rossi studied a variant of the tug-of-war game and its connection to the dynamic programming principle (DPP)
$$u_{\eps}(x)=\frac{\alpha}{2}(\sup_{B_{\eps}(x)} u_{\eps}+\inf_{B_{\eps}(x)} u_{\eps}  ) +\beta \vint_{B_\eps(x)} u_\eps(z)\, dz,$$
where $u_{\eps}$ denotes the value of the game,  $\alpha$ and $\beta$ are given probabilities, and $\eps>0$ denotes the upper bound for the step size.
Roughly,  at each round either the game position moves to a random point with probability $\beta$, or with probability $\alpha$ the players toss a coin and the winner of the toss decides where to move. The game is played in a domain $\Omega$, and once the game position exits the domain, Player II pays Player I the amount given by a payoff function.
As $\eps\to 0$, the value functions converge to the corresponding $p$-harmonic function with suitable choices of $\alpha$ and $\beta$.
The game in \cite{manfredipr12} has good symmetry properties, and this allows a rather straightforward proof of Lipschitz continuity \cite{luirops13} of $p$-harmonic functions. The proof is based on a suitable choice of strategies and is thus quite different from the PDE proofs. 

In this paper we study a different version of the game where we randomize the step size  for the  tug-of-war part, that is, (upper bound for) the step size of the players is chosen according to the uniform distribution on $[0, \eps]$.  We give a detailed description of the game in Section \ref{sect2}. The key outcome is that,  randomizing the step size for the tug-of-war part has a regularizing effect on the value function.  We will also show that the game  has a value and that the value function satisfies the following DPP
\begin{align*}
u_{\eps}(x)=\frac{\alpha}{2\eps}\int_{0}^{\eps} (\sup_{B_{t}(x)} u_{\eps}+\inf_{B_{t}(x)} u_{\eps}  )\ud t +\beta \vint_{B_\eps(x)} u_\eps(z)\, dz.
\end{align*}
In one of our main results, in Theorem \ref{thm:main}, we show almost everywhere that the gradient of the value function $u_{\eps}$ exists and is bounded. As in the standard tug-of-war with noise, the value functions converge uniformly to the corresponding $p$-harmonic function as the step size tends to zero, but now also the gradients converge weakly to the gradient of the $p$-harmonic functions as stated in Theorem \ref{thm:conv-of-gradients}. In order to obtain the existence and boundedness of the gradient in Theorem \ref{thm:main}, we need to control the small scale behavior of the value function. This is missing in the standard tug-of-war game and the value can even be discontinuous. However, when randomizing over the step size there is a considerable overlap in the  small scale and thus we can establish cancellation effect, see the estimate \eqref{eq:key-overlap}.

The sharper Lipschitz estimate when boundary values are close to a plane is obtained in Theorem \ref{thm:lip-est}. The key idea is to modify the cylinder walk argument introduced in \cite{luirops13} so that boundary values are encoded into the cylinder walk. Moreover, the modified cylinder walk   directly gives an estimate for the oscillation of the value function. 
 
More regular, and in particular continuous, versions of tug-of-war type games have been suggested by Lewicka in \cite{lewicka}. Despite the lack of infinitesimal regularity of the standard tug-of-war type game, its regularity can be studied asymptotically. For asymptotic Hölder and Lipschitz regularity results see, in addition to the references mentioned above, for example \cite{ruosteenoja16, arroyohp17, luirop18,  arroyolpr}. We are mostly interested in the regularity theory of games on its own right, but mention that as an application our regularity results for games imply new proofs for regularity results for $p$-harmonic functions and the corresponding numerical discretization schemes.


\section{Randomized step size game}\label{sect2}
Consider a bounded domain $\Omega\subset\R^n$  satisfying the uniform exterior sphere condition and let $\eps \in (0,1)$. We denote the compact $\eps$-boundary strip by 
$$\Gamma_\eps:= \left\{ x\in \R^n\setminus \Omega\, :\, \textrm{dist}(x, \partial \Omega)\leq \eps\right\}.$$
We also set
$$\Omega_\eps:=\Omega\cup \Gamma_\eps.$$
Here and subsequently, we denote by $B_{t}(x)$ the open ball of radius $t$ centered at $x$.
We assume that  $n\geq 2$ and  $2<p<\infty$. Here $p$ is related to the $p$-Laplacian in the limiting problem.

\subsection{Rules of the game}

We define  a  variant of tug-of-war with noise that we call {\it random step size TWN} played by Player I and Player II as follows.  First, a token is placed at a point  $x_0\in \Omega$ and the players toss a biased coin with probabilities $$
\alpha=\frac{p-2}{n+p}\in (0,1)
\quad \text{and}\quad  \beta=\frac{n+2}{p+n}=1-\alpha.$$
If they get tails (probability $\beta$), the game state moves randomly (according to the uniform distribution) to a  point $x_1$ in the ball $B_\eps(x_0)$.
If they get heads (probability  $\alpha$),  a step size $\eps_1$ is chosen randomly on $[0, \eps]$ (according to the uniform distribution) and a fair coin is tossed,
then the winner of the toss is allowed to move the game position to any  point $x_1\in B_{\eps_1}(x_0)$. They continue
playing according to the same rules at $x_1$. The game continues until the token hits $ \Gamma_\eps$ for the first time, and Player II pays Player I the amount $F(x_\tau)$.
The point
$x_\tau$ denotes the first point outside the domain $\Omega$ and $\tau$ refers to the first time we hit  $\Gamma_\eps$.
The payoff function $F :  \Gamma_\eps \to \R$ is a given, bounded, and Borel measurable function.
Player I attempts to maximize the payoff, while Player
II attempts to minimize it.
A history of the game up to step $k$ is given by a vector 
$$( x_0, (c_1, \eps_1, x_1), (c_2, \eps_2,  x_2),\ldots, (c_k, \eps_{k}, x_k))$$
with
\begin{itemize}
\item coin tosses 
$c_i\in \left\{0,1,2\right\}$ where 1 denotes that Player I wins, 2 that Player II wins
and 0  that a random step occurs,\\
\item the step sizes $\eps_i\in [0, \eps]$,\\
\item the  game states  $x_i$.
\end{itemize}
We associate to the  history of the game the filtration $\left\{\F_k\right\}_{k=0}^\infty$, where $\F_0 :=\sigma(x_0)$ and
for $k\geq 1$
$$\F_k :=\sigma(x_0, (c_1, \eps_1, x_1), (c_2, \eps_1,  x_2),\ldots, (c_k, \eps_{k}, x_k)).$$
A strategy 
for Player I that we denote for short $S_\I$ is a collection of measurable functions (with respect to a suitable filtration $\F_k'$)
that give the next game position given the history of the game and the next step size, that is
$$S_\I(( \eps_1, x_0), (c_1, \eps_2, x_1), (c_2, \eps_3,  x_2),\ldots, (c_k, \eps_{k+1}, x_k))= x_{k+1}\in B_{\eps_{k+1}}(x_k)$$
if Player I wins the toss. Similarly Player II uses a strategy $S_{\J}$.
The rules of the game give one step probability measures. Using this, with the fixed starting point $x_0$ and the strategies $S_\I$ and $S_\J$, we can construct a
unique probability measure $\P^{x_0}_{S_\I, S_\J}$ on the game trajectories. 

Let $S_{\text{I}}$ be the strategy for the first player and $S_{\text{II}}$  the strategy for the second player. We define the value of the game  for  Player I as
$$u^\eps_{\text{I}}(x_0):=\underset{S_{\text{I}}}{\sup}\, \underset{S_{\text{II}}}{\inf}\,  \E^{x_0}_{S_{\text{I}}, S_{\text{II}}} [F(x_\tau)],$$
and the value of the game for Player II as 
$$u^\eps_{\text{II}}(x_0):=\underset{S_{\J}}{\inf}\, \underset{S_\I}{\sup}\,  \E^{x_0}_{S_{\text{I}}, S_{\text{II}}} [F(x_\tau)].$$
Due to the fact that $\beta>0$, the game
ends almost surely
for any choice of strategies.
\subsection{The DPP and the comparison principle}

An important property of value functions of tug-of-war type games is the dynamic programming principle (DPP). Using similar arguments as in  \cite{luirops14}, we can show that the game  has a value and that  the value function satisfies the following DPP.
\begin{lemma}[Existence, uniqueness and the DPP]
There exists a unique value function 
$$u_\eps:=u_{\I}^\eps=u_{\J}^\eps$$
 in $\Omega_\eps= \Omega\cup\Gamma_\eps$ satisfying  $u_\eps=F$ on $\Gamma_\eps$. Moreover $u_\eps$ satisfies the DPP
\begin{align}\label{dpp1}
u_{\eps}(x)=\frac{\alpha}{2\eps}\int_{0}^{\eps} (\sup_{B_{t}(x)} u_{\eps}+\inf_{B_{t}(x)} u_{\eps}  )\ud t +\beta \vint_{B_\eps(x)} u_\eps(z)\, dz.
\end{align} 

\end{lemma}
A slight modification of the arguments used in \cite{luirops14} implies the existence and uniqueness of the bounded function satisfying the DPP \eqref{dpp1} and taking boundary values $F$ on $\Gamma_\eps$. This function can then be used to show that the game has a value i.e.\ $u_{\I}^\eps=u_{\J}^\eps$, see  Appendix \ref{appendA}.

The proof of the previous lemma also gives us the following comparison principle.

\begin{proposition}\label{popcomp}
Let $\bar u$ be a bounded function satisfying the DPP \eqref{dpp1} and such that $\bar u\geq F$ in $\Gamma_\eps$, and $u_{\eps}$ the value of the game with boundary values $F$. Then it holds
$$\bar u\geq u_\eps\quad\text{in}\quad \Omega_\eps.$$
Similar result also holds if the inequalities are reversed.
\end{proposition}
From the comparison principle,  we get the uniform boundedness  of $u_\eps$.
\begin{lemma}
Let $u_\eps$ be the value function of the random step size TWN with boundary values $F$ on $\Gamma_\eps$. Then it holds
$$|u_\eps(x)|\leq \underset{\Gamma_\eps}{\sup}\, |F|\qquad \text{for}\quad x\in \Omega_\eps.$$

\end{lemma}

\section{Existence, boundedness and weak convergence of the gradient}\label{sect3}
In order to obtain  a Lipschitz estimate independent of $\eps$, we proceed in two steps. First we provide a large scale estimate that has  an $\eps$-dependent error using a cylinder walk method  introduced in  \cite{luirops13}. Then we utilize overlap and cancellation in the small scale to improve the estimate.   
In the sequel $C$  will denote a generic constant which may change
from line to line.
\begin{lemma}\label{lipestrough1}
Let $u_{\eps}$  be the value function of the random step size TWN with  boundary values $F$.  Assume that $B_{6r}(z_0)\subset\Omega$ with $r>\eps$. Then  there exists a constant $C>0$ depending only on $p,n, r$ \ and $\norm{F}_{L^\infty(\Gamma_\eps)}$
 such that,   for $x,y\in B_r(z_0)$, it holds
\begin{equation}\label{roughlipesteps}|u_\eps(x)-u_\eps(y)|\leq  C|x-y|+C\eps.
\end{equation}

\end{lemma}
\begin{proof}
 {\bf Step 1: Cancellation strategy.}
Given two points $x, y\in B_r(z_0)$ with $B_{4r}(z_0)\subset \Omega$, we  fix a point $z$ such that
$$|x-z|=|y-z|=|x-y|/2.$$
 Suppose that the game starts at $x$. At every step $k$  we can describe the game position as a sum of vectors
\begin{align}
\label{eq:game-posit}
x+\sum_{j\in J_1^k} v_j+\sum_{j\in J_2^k} w_j+\sum_{j\in J_3^k} h_j.
\end{align}
Here $J_1^k$ denotes the indexes of rounds when Player I has moved, vectors $v_j$
are her
moves, and $J_{2}^k$  denotes the indexes of rounds when Player II has moved,  the $w_j$
represent the moves of Player II. The set $J_3^k$ denotes
the indexes when we have taken a random move, and these vectors are denoted by $h_j$.
Let  us define a strategy $S_{\J}^0$
for Player II for the game that starts from $x$. Player II always tries to cancel the moves of Player I which he has not yet been able to cancel and otherwise he moves to the direction $z-x$  with the aim
$$x+\sum_{j\in J_1^k} v_j+\sum_{j\in J_2^k} w_j=z.$$

What we mean by ''cancellation''  is that Player II is backtracking the path made by Player I and going towards $z$. Since the player is allowed to step to a point within an open ball, we will have to choose a radius slightly smaller that $\eps_k$ that is not accumulating too much errors.  More precisely,  if at  the $(k+1)$-step Player II wins the coin toss, then first he looks at the  remaining part of the track made by Player I  that he  has not yet been able to cancel,  let's call it by $$V_k=\sum_{j\in J_k^1}v_j- E_k,$$
where $E_k$  is the canceled part.
Suppose  $V_k\neq 0$, then $x_k-V_k$ is inside or outside the ball  $ B(x_{k}, (1-2^{-k-1})\eps_{ k+1})$. If $x_k-V_k$ is outside the ball $B(x_{k}, (1-2^{-k-1})\eps_{ k+1})$ then Player II moves to a point $x_{k+1}$ which  is the intersection of the ball $\bar B(x_{k}, (1-2^{-k-1})\eps_{ k+1})$ with the line going from $x_k$ to $x_k-V_k$. If $x_k-V_k\in B(x_{k}, (1-2^{-k-1})\eps_{ k+1})$ then Player II  moves to a point $x_{k+1}$ such that $x_{k+1}=x_k-V_k+\lambda_k (z-x)$ where $\lambda_k$ is either 0 if 
$$x+\sum_{j\in J_1^k} v_j+\sum_{j\in J_2^k} w_j-V_k=z,$$
or $\lambda_k>0$ is such that 
$x_{k+1}=x_k-V_k+\lambda_k (z-x)\in \partial B(x_k, (1-2^{-k-1})\eps_{k+1})$.
  If all the moves of Player I at that moment are canceled that is $V_k=0$ and Player II
wins the coin toss, then he moves to the direction $z-x$ by  vectors of the form $c_k \dfrac{(z-x)}{|z-x|}$  where $c_k = (1-2^{-k-1})\eps_{k+1}$ as long as 
$$\abs{x+\sum_{j\in J_1^k} v_j+\sum_{j\in J_2^k} w_j-z}\ge  (1-2^{-k-1})\eps_{k+1}. $$ If not, we choose a suitable $0\leq c_k \leq (1-2^{-k-1})\eps_{k+1}$, 
so that 
$$x+\sum_{j\in J_1^k} v_j+\sum_{j\in J_2^k} w_j=z.$$

We stop this process if one of the following conditions holds: 
\begin{enumerate}[\bf C1)]
\item  \begin{equation}\label{leaderdist}\sum_{j=0}^i \eps_j a_j< -|x-z|-\eps,
\end{equation}
where  for $j\geq 1$ we set $a_j=1$ if Player I wins at the $j$-th step, $a_j=-1$ when Player II wins and $a_j=0$ if the random move occurs, and $a_0=0$. The quantities $\eps_j$ are the (upper bounds of)  step sizes of the game.
\item 
 $$\sum_{j=0}^i \eps_j a_j\geq r,$$
\item $$|\sum_{j\in J_{3}^i} h_j|>r.$$
\end{enumerate}

We define  $\tau'$ as the stopping time defined by those conditions. With probability 1 this stopping time is finite. An important point to note here is that this stopping time does not depend on the
strategies.
Notice that when the game has ended by condition  {$\bf C1$} and one of the player is using the cancellation strategy, then the final point $x_{\tau'}$  is randomly chosen around $z$:
$$x_{\tau'}=z+\sum_{j\in J_3^{\tau'}} h_j.$$

The condition ${\bf C1}$ guarantees that the Player II has played sufficiently many turns with sufficient step sizes to place the token at $ z$ (modulo the random noise). Indeed notice that when using the cancellation strategy (since we either add a vector that cancels the moves of the other player or add vectors of the form $c(z-x)$) we have
$$
x_{\tau'}=z+x-z+\sum_{j\in J_1^k} v_j+\sum_{j\in J_2^k} w_j+\sum_{j\in J_3^k} h_j=z + H+ \sum_{j\in J_3^k} h_j$$
 with 
 \begin{align*}|H|&\leq \max\left(0, |x-z|+\sum_{j \, \text{such that},  a_j=1} a_j\eps_j+ \sum_{j \, \text{such that}\, a_j=-1} a_j\eps_j(1-2^{-j-1})\right)\\
 &\leq\max\left(0, |x-z|+\sum_j a_j\eps_j+\eps \sum_{j} 2^{-j-1}\right)\leq 0,
 \end{align*}
 i.e.\ $H=0$.
We can utilize the cancellation effect by using the symmetry of this construction.
Letting $S_\I^0$ be the corresponding  cancellation strategy for Player I when starting from the point $y$, it holds 
$$\E^x_{S_\I, S^0_{\J}}[u_\eps(x_{\tau'})|\text{game ends by  \bf C1}]=\E^y_{S_\I^0, S_{\J}}[u_\eps(x_{\tau'})|\text{game ends by  \bf C1}]$$
for  any choices of the strategies $S_\I, S_{\J}$. Hence we can eliminate the symmetric part when estimating $u_\eps(x)-u_\eps(y)$.
Also observe that in all cases, we are guaranteed  that, when the game is still running, we never exit $B_{4r}(z_0)$. 
By an abuse of notation, for $i\in \left\{1,2,3\right\}$,  we denote by $\text{C}_i$ the event the game ends by condition {\bf Ci}. We have that
 \begin{align}\label{huli}
|u_\eps(x)-u_\eps(y)|&\leq\underset{S_\I, S_\J}{\sup} \left|\E^x _{S_\I, S_{\J}^0}[u_\eps(x_{\tau'})]-\E^y_{S_\I^0, S_{\J}}[u_\eps(x_{\tau'})]\right|\nonumber\\
&\leq (1-P)\underset{S_\I, S_\J}{\sup} \left|\E^x _{S_\I, S_{\J}^0}[u_\eps(x_{\tau'})|\text{C}_2 \,\textrm{or}\,  \text{C}_3]-\E^y_{S_\I^0, S_{\J}}[u_\eps(x_{\tau'})|  \text{C}_2 \,\textrm{or}\,  \text{C}_3]\right|\nonumber\\
&\leq 2(1-P) \norm{u_\eps}_{L^\infty(\Omega_\eps)},
\end{align}
where $P$ denotes the probability that the process ends by {\bf C1}.

{\bf Step 2: Cylinder walk.} We associate to this process a cylinder walk.  Consider the following random walk in a $n+1$-dimensional cylinder $B_r(0)\times(0, r+|x-z|+\eps)$.
Rules of the walk:
\begin{itemize}
\item the token is initially at $(\zeta_0,t_0)=(0, |x-z|+\eps)$,
\item  with probability $\alpha$,  the token moves randomly from $(\zeta_j, t_j)\in B_r(0)\times(0, r+|x-z|+\eps)$  to $(\zeta_{j+1}, t_{j+1})$,
where $t_{j+1}$ is chosen according to a uniform probability distribution on $[t_j-\eps, t_j+\eps]$, and $\zeta_{j+1}:=\zeta_{j}$,

\item with probability $\beta$ the token moves randomly from $(\zeta_j, t_j)$  to  $(\zeta_{j+1}, t_{j+1})$ where $\zeta_{j+1}$ is chosen according to a uniform probability distribution on $B_\eps(\zeta_j)$, and $t_{j+1}:=t_j$.
\end{itemize}

The idea is that we associate the $t$-component of the cylinder walk process with the random variable
\begin{equation}\label{mischka}
\sum_{j=0}^{i}\eps_j a_j+|x-z|+\eps,
\end{equation}
where $a_j$ are defined as in \eqref{leaderdist}. Similarly, we associate $x$-position of the cylinder walk process with 
\begin{align*}
\sum_{j\in J_{3}^i} h_j
\end{align*}  
in (\ref{eq:game-posit}).
Each of the stopping conditions in the original process is associated with reaching the boundary in the cylinder walk. 

By using a function   $\bar v$ satisfying 
\begin{enumerate}
\item 
\begin{equation}\label{equatest1}
\dfrac{(p-2)}{3}\bar v_{tt}+\Delta \bar v=0 \quad\text{in}\quad B(0, 4r)\times [-r,  5r],
\end{equation}
\item $\bar v\leq 0$ on the sides and the top of the cylinder,\\
 \item $\bar v(x, -\eps)=1$  for $x\in B(0,r/2)$,
\end{enumerate}
we can prove that 
\begin{equation}\label{eq:probab}
1-P\leq C|x-y|+C\eps.
\end{equation}
To see this,   we use the Taylor expansion and the fact that $\bar v$ satisfies \eqref{equatest1},  and we get 
$$\bar v (x,t)=\beta\vint_{B_\eps(x)} \bar v(y,t) \,dy+\frac{\alpha}{2\eps}\int_{t-\eps}^{t+\eps} \bar v(x,s) \,ds+ c(n)||D^3 \bar v||\eps^3.$$
It follows that,  if we consider the sequence of random variables $\bar v(x_j , t_j )$,
$j =0, 1, 2, . . . $, with  $(x_j , t_j )_{j\in \N}$ being  the positions in the cylinder walk, then  we have
\begin{align*}
\E^{x_0, t_0}\left[\bar v(x_{k+1}, t_{k+1})|\mathcal{F}_{k}\right]&=\frac{\alpha}{2\eps}\int_{t_{k}-\eps}^{t_{k}+\eps} \bar v(x_k,s) \,ds+\vint_{B_\eps(x_{k})}\bar v (z, t_k)\, dz\\
&\geq  \bar v (x_{k}, t_k)- C\eps^3.
\end{align*}
Consequently,  $M_j := \bar v(x_j , t_j ) +Cj\eps^3$ is a submartingale for a constant $C$ depending on $n$ and $|D^3\bar v|$. Now we apply the optional
stopping, using the stopping time $\bar \tau$
that corresponds to exit from the domain $B_{r}(0)\times [0, |x-z|+r+\eps]$. We get that
$$1-|\bar v_t|(t_0+\eps)\leq \bar v(0, t_0)=\E[M_0]\leq\E[\bar v(x_{\bar \tau}, t_{\bar \tau})]+C\eps^3 \E[\bar \tau],$$
and  by the convexity of $\bar v$, we have
$$\E[\bar v(x_{\bar \tau}, t_{\bar \tau})]\leq P\underset{x\in B_r(0)}{\sup} \bar v(x, -\eps)+O(\eps)\leq P+O(\eps),$$
where $O(\eps)$ is an error coming from not stopping on the boundary of the cylinder. A slight modification of the reasoning  of \cite[Appendix A]{luirops13}, see Appendix \ref{sec:stochastic-proof-random}, gives us  the estimate  $\E[\bar \tau]\leq C/\eps^2$.
Remembering that $t_0=|x-z|+\eps,$ we get that
$1-P\leq C(|x-z|+\eps)+C\eps$. This estimate together with \eqref{huli}, complete the proof.
\end{proof}
Next, combining the previous result and a small scale overlap, we can prove the existence and boundedness of the gradient for value functions. 
\begin{theorem}
\label{thm:main}
 Let $u_\eps$ be the value function to the random step size TWN  with boundary values $F$. Assume that $B_{5r}(z_0)\subset\Omega$ with $r>\eps$. Then,  there exists a  constant $C>0$ depending on $p,n, r$ and $ \norm{u_\eps}_{L^\infty(\Omega_\eps)}$   such that,  for $x, y\in B_{2r}(z_0)$, it holds
 \begin{equation}\label{lipou}
 |u_\eps(x)-u_\eps(y)|\leq C|x-y|.
 \end{equation}
 Moreover,  $Du_\eps$ exists almost  everywhere in $B_r(z_0)$ and
 $$|Du_\eps(z)|\leq C,\qquad \text{a.e in}\quad B_r(z_0).$$

\end{theorem}

\begin{proof} 
Fix  $x,y\in B_{2r}(z_0)$.   If $|x-y|\geq \eps$, then the estimate \eqref{lipou} follows from \eqref{roughlipesteps}, and thus we may focus our attention to the case $|x-y|\leq \eps$.   Using the DPP formulation, we have
\begin{align}
\label{eq:start-up}
\dfrac{u_\eps(x)-u_\eps(y)}{|x-y|}&=\dfrac{\alpha}{2\eps |x-y|}\int_{0}^{\eps} (\sup_{B_{t}(x)} u_{\eps}+\inf_{B_{t}(x)} u_{\eps}  )\ud t +\dfrac{\beta}{|x-y|} \vint_{B_\eps(x)} u_\eps(z)\, dz\nonumber\\
&-\dfrac{\alpha}{2\eps |x-y|}\int_{0}^{\eps} (\sup_{B_{t}(y)} u_{\eps}+\inf_{B_{t}(y)} u_{\eps}  )\ud t -\dfrac{\beta}{|x-y|} \vint_{B_\eps(y)} u_\eps(z)\, dz.
\end{align}
Since $|x-y|$ is small we can utilize the overlap between the balls and benefit from the resulting cancellations.
We treat the tug-of-war part and the random noise part in a different manner.

\noindent {\bf Step 1: Tug-of-war part.} Define
$$G(x,y):=\frac{\alpha}{\abs{x-y}}\Big(\frac{1}{2\eps}\int_{0}^{\eps} (\sup_{B_{t}(x)} u_{\eps}+\inf_{B_{t}(x)} u_{\eps}  )\ud t
-\frac{1}{2\eps}\int_{0}^{\eps} (\sup_{B_{t}(y)} u_{\eps}+\inf_{B_{t}(y)} u_{\eps}  )\ud t\Big).$$
We rearrange $G$ as
\begin{align*}
G(x,y)
&=\frac{\alpha}{2\eps \abs{x-y}}\Big\{\Big(\int_0^\eps \sup_{B_{t}(x)} u_{\eps}\ud t-\int_0^\eps \sup_{B_{t}(y)} u_{\eps} \ud t\Big)\\
&\hspace{7 em}+\Big(\int_0^\eps \inf_{B_{t}(x)} u_{\eps}\ud t-\int_0^\eps \inf_{B_{t}(y)} u_{\eps}  \ud t\Big)\Big\}\\
&=I+J.
\end{align*}
We start by an estimate for $I$
\begin{align}
\label{eq:key-overlap}
I=&\frac{\alpha}{2\eps \abs{x-y}}\Big\{\int_0^{\eps-\abs{x-y}} \sup_{B_{t}(x)} u_{\eps}\ud t-\int_{\abs{x-y}}^\eps \sup_{B_{t}(y)} u_{\eps} \ud t\nonumber \\
&\hspace{6 em}+\int_{\eps-\abs{x-y}}^{\eps} \sup_{B_{t}(x)} u_{\eps}\ud t-\int_{0}^{\abs{x-y}} \sup_{B_{t}(y)} u_{\eps} \ud t\Big\}\nonumber \\
&=\frac{\alpha}{2\eps \abs{x-y}}\Big\{\int_{0}^{\eps-\abs{x-y}}\underbrace{(\sup_{B_{t}(x)} u_{\eps}-\sup_{B_{t+\abs{x-y}}(y)} u_{\eps} )}_{\leq 0}\, dt\\
&\hspace{8 em}+\int_{0}^{\abs{x-y}} (\sup_{B_{\eps-t}(x)} u_{\eps}-\sup_{B_{t}(y)} u_{\eps} )\ud t\Big\}.\nonumber 
\end{align}
Here we used that $B_{t}(x)\subset B_{t+\abs{x-y}}(y)$.
Next we estimate the second term in (\ref{eq:key-overlap})
by using the result of Lemma \ref{lipestrough1}. We have
\begin{align*}
\big(\underset{B_{\eps-t}(x)}{\sup} u_{\eps}-\underset{B_{t}(y)}{\sup} u_{\eps} \big)&\le \sup_{B_{\eps-t}(x)} u_\eps-u_\eps(y)\\
&\leq \sup_{z\in B_{\eps-t}(x)} (u_\eps(z)-u_\eps (y))\\
&\leq  C\sup_{z\in B_{\eps-t}(x)}|z-y|+ C \eps\\
&\leq C (\eps+\abs{x-y}).
\end{align*}
 It follows that
 $$
 \abs I\leq \frac{\alpha C}{2\eps}(\eps+\abs{x-y})\leq\alpha C,
 $$
for some $ C>0$ that depends on $p,n,r$  and  $\norm{ F}_{L^\infty(\Gamma_\eps)}$.
Similarly for $J$,  we have
\begin{align*}
J=&\frac{\alpha}{2\eps \abs{x-y}}\Big\{\int_0^{\abs{x-y}} \inf_{B_{t}(x)} u_{\eps}\ud t-\int_{\eps-\abs{x-y}}^\eps \inf_{B_{t}(y)} u_{\eps} \ud t\\
&\hspace{6 em}+\int_{\abs{x-y}}^{\eps} \inf_{B_{t}(x)} u_{\eps}\ud t-\int_{0}^{\eps-\abs{x-y}} \inf_{B_{t}(y)} u_{\eps} \ud t\Big\}\\
&=\frac{\alpha}{2\eps \abs{x-y}}\Big\{\int^{\eps}_{\abs{x-y}}\underbrace{(\inf_{B_{t}(x)} u_{\eps}-\inf_{B_{t-\abs{x-y}}(y)} u_{\eps} )}_{\leq 0}\, dt\\
&\hspace{8 em}+\int_{0}^{\abs{x-y}} (\inf_{B_{t}(x)} u_{\eps}-\inf_{B_{\eps-t}(y)} u_{\eps} )\ud t\Big\}.
\end{align*}
Then, using the result of Lemma \ref{lipestrough1}, we estimate
\begin{equation*}
\inf_{B_{t}(x)} u_{\eps}-\inf_{B_{\eps-t}(y)} u_{\eps} \leq
C(\eps+|x-y|).
\end{equation*}
In the same way
$$\abs{J}\leq \frac{C\alpha}{2\eps}(\eps+\abs{x-y})\leq C \alpha.$$
Combining the estimates for $I$ and $J$, we get that
\begin{align}
\label{gampart}
&\abs{G(x,y)} \leq 2\alpha  C.
\end{align}

\noindent{\bf Step 2:  Random part.}
Here we want to estimate
$$H:=\beta\left(\vint_{B_\eps(x)} u_\eps(z)\, dz- \vint_{B_\eps(y)} u_\eps(z) \, dz\right),$$ which arises from (\ref{eq:start-up}).
Recall  that $|x-y|\leq\eps$. We  fix a point $\bar h\in \partial (B_\eps(x)\cap B_\eps(y))$. We have
\begin{align*}
\vint_{B_\eps(x)} u_\eps(z)\, dz&- \vint_{B_\eps(y)} u_\eps(z) \, dz\\
&=\frac{1}{|B_\eps(0)|}\left[\int_{B_\eps (x)\setminus (B_\eps(x)\cap B_\eps(y))} u_\eps(z)\, dz-\int _{B_\eps(y)\setminus (B_\eps(x)\cap B_\eps(y))}u_\eps(z)\,dz\right]\\
&=\frac{1}{|B_\eps(0)|}\int_{B_\eps (x)\setminus (B_\eps(x)\cap B_\eps(y))} \underbrace{\left(u_\eps(z)-u_\eps(\bar h)\right)}_{A_1}\, dz\\
&\quad+\frac{1}{|B_\eps(0)|}\int_{B_\eps (y)\setminus (B_\eps(x)\cap B_\eps(y))} \underbrace{\left(u_\eps(\bar h)-u_\eps(z)\right)}_{A_2}\, dz.
\end{align*}
 Using the estimate coming from Lemma \ref{lipestrough1}, we have 
 $$|A_i|\leq C|z-\bar h|+C \eps\leq C\eps.$$
Similarly, it holds
 \begin{align}\label{randes}
 \abs{H}&\leq C\beta\eps\frac{2|B_\eps (x)\setminus (B_\eps(x)\cap B_\eps(y))|}{|B_\eps(0)|}\nonumber\\
 &\leq\beta C |x-y|.
 \end{align}
 
 Here we used that $$|B_\eps (x)\setminus (B_\eps(x)\cap B_\eps(y))|\leq |x-y|\omega_n\eps^{n-1}=\frac{n}{\eps}|x-y||B_\eps(0)|,$$
 where $\omega_n$ is  the  surface area of the $(n-1)$-dimensional unit sphere.  Summing the estimates \eqref{gampart} and \eqref{randes}, we get that 
 \[
 \frac{|u_\eps(x)-u_\eps(y)|}{|x-y|}\leq C,
  \]
 for all $x, y\in B_{2r}(z_0)$.
 \end{proof}

 We are now in a position to show the weak convergence of the gradient and the relation to $p$-harmonic functions.
 For the theory of $p$-harmonic functions, see for example \cite{heinonenkm93} or \cite{lindqvist06}. These references mostly deal with the weak theory of partial differential equations. The tug-of-war approach leads to the viscosity solutions of the normalized $p$-Laplacian, but in the homogeneous case these solutions  coincide with the usual $p$-harmonic functions \cite{juutinenlm01,kawohlmp12}.  
\begin{theorem}
\label{thm:conv-of-gradients}
Let $F\in C(\Gamma_\eps)$, $2<p<\infty$ and let $u_\eps$ be the value function of the  random step size TWN with  boundary values  $F$.  Assume that $\Om$ satisfies a uniform exterior sphere condition.  
Let $u$ be the unique $p$-harmonic function in $\Omega$ with $u=F$ on $\partial\Omega$. Then 
\begin{align*}
u_\eps&\to u\qquad &\text{uniformly on}\,\,\ol \Omega,\\
\end{align*}
and for any $q\in [1, \infty)$ and $B_{2r}(z_0)\subset\Omega$ it holds up to a subsequence that
\begin{align*}
Du_\eps&\rightharpoonup Du\quad&\text{weakly in}\,\, L^q(B_r(z_0)).
\end{align*}
\end{theorem}
\begin{proof}
  From Theorem \ref{thm:main}, we  know that for $B_{2r}(z_0)\subset\Omega$,  there exists a constant $C$ independent of $\eps$ such that $\norm{Du_\eps}_{L^\infty(B_{2r}(z_0))}\leq C$.  First, a straightforward   modification of the arguments used in \cite{manfredipr12}  allows us to prove  that as $\eps\to 0$, the value functions  converge  uniformly to the unique $p$-harmonic function $u$  in $\Omega$ with $u = F $ on $\partial \Omega$. For the convenience of the reader, we work out the details in the Appendix \ref{appendixB}.
 
 The weak convergence of a subsequence in the Sobolev spaces $W^{1,q}(B_r(z_0))$ for $1<q<\infty$ also follows from the above estimate since it implies that the sequence is uniformly bounded in these reflexive spaces. The case $q=1$ follows from the equi-integrability of $Du_\eps$ and the  Dunford-Pettis theorem.
\end{proof}

\section{Lipschitz estimate}\label{sect4}

In this section we provide a sharper Lipschitz estimate for the value functions $u_\eps$ when we have additional knowledge about the boundary values. 
If  the boundary function is relatively close to a plane, does the  Lipschitz estimate of the value function stay close to the slope of the linear function inside the domain. 
 This is related to the strong convergence in Sobolev spaces, see for example \cite[Theorem 4.1]{evanss11}. However, due to some subtle errors we could not reach a quite sufficient estimate $|Du_\eps|\leq |\nu|+C\delta$. 

First we state immediate bounds arising from the comparison with planes. 
\begin{lemma}\label{lemcomp}
Let $\nu\in \R^n,  b\in \R$ and $\delta>0$.  Assume that $F$ is a continuous function which satisfies in  $\Gamma_\eps$ 
\begin{align*}
\abs{F(x)-\nu \cdot x-b}\le \delta.
\end{align*}
Let $u_{\eps}$ be  the value function  for the random step size TWN with  boundary values $F$. Then  for $x\in \Omega_\eps$, we have
\begin{align*}u_\eps(x) &\leq \bar u(x):=\nu\cdot x+\delta+b,\\
u_\eps(x) &\geq\underline u(x):= \nu\cdot x-\delta+b.
\end{align*}
\end{lemma}
\begin{proof}
Since $\bar u$  satisfies the DPP \eqref{dpp1}
and $\bar u\geq F$,  the comparison principle of Proposition \ref{popcomp} implies that $u_\eps (x)\leq \bar u$ for  $x\in \Omega_\eps$.
The same argument implies that $u_\eps (x)\geq \underline u$ for $x\in \Omega_\eps$.
\end{proof}


\begin{theorem}\label{thm:lip-est}
Let $\nu\in\R^n$, $b\in\R$  and  $F\in C(\Gamma_\eps)$. Assume that
\begin{align*}
\abs{F(x)-\nu\cdot x-b}\le \delta.
\end{align*}
Let $u_{\eps}$  be the value function for the random step size TWN with boundary values $F$.  Assume that $B_{6r}(z_0)\subset\Omega$. Then,   there exists a constant $C>0$ depending only on $p,n$ and $r$ such that,   for $x,y\in  B_r(z_0)$, it holds
\[
|u_\eps(x)-u_\eps(y)|\leq (|\nu|+C\delta)|x-y|+(5|\nu|+C \delta)\eps.
\]
\end{theorem}

\begin{proof} 
 The key idea is to consider again the cancellation strategy,  the previous  stopping rules for  the associated cylinder walk but to use a different  barrier function  (that we will construct explicitly) which  directly gives an estimate for the difference of values and thus immediately gives  Lipschitz estimate for the value function instead of just giving  an estimate for the hitting probabilities. Such technique should be of independent interest. The proof will be divided into 4 steps.
 
  {\bf Step 1: Cancellation strategy and properties of $u_\eps$.}  
  Given two points $x, y\in B_r(z_0)$ with $B_{4r}(z_0)\subset \Omega$, we  fix a point $z$ such that
$$|x-z|=|y-z|=|x-y|/2.$$
We define the same cancellation strategy as in Section \ref{sect3} and the same stopping rules ${\bf C1, C2, C3}$. 
For $i\in \left\{1,2,3\right\}$,  we denote by $\text{C}_i$ the event the game ends by condition {\bf Ci}.
Using the cancellation strategies we have again:
 \begin{equation}\label{cancelsym}
 E^x_{S_\I, S^0_{\J}}[u_\eps(x_{\tau'})|\text{C}_1]=\E^y_{S_\I^0, S_{\J}}[u_\eps(x_{\tau'})|\text{C}_1].
 \end{equation}

 Next we can write using a shorthand $P:=P^x_{S_\I, S^0_{\J}}$
 \begin{align*}
 \E^x_{S_\I, S^0_{\J}}[u_\eps(x_{\tau'})]&=P( \text{C}_1)\, \E^x_{S_\I, S^0_{\J}}[u_\eps(x_{\tau'})|\text{C}_1]+P(\text{C}_2)\,\E^x_{S_\I, S^0_{\J}}[u_\eps(x_{\tau'})|\text{C}_2]\\
 &\quad+ P(\text{C}_3)\,\E^x_{S_\I, S^0_{\J}}[u_\eps(x_{\tau'})|\text{C}_3],
 \end{align*}
 where $P(\text{C}_i)$ are independent of strategies.
 In the sequel we will the notation used in \cite[Section 8.3]{klenke2014} 
for conditional expectations.
We introduce the random variable $Y$  taking values in $\R$  by   
$$
Y=\sum^{\tau'}_{j=0} a_j\eps_j
$$
and write
\begin{align}
P(\text{C}_3)\E^x_{S_\I, S^0_{\J}}[u_\eps(x_{\tau'})|\text{C}_3]&= \E^x_{S_\I, S^0_{\J}}[u_\eps(x_{\tau'})\mathbbm{1}_{\text{C}_3}]\nonumber\\
&=\E^x_{S_\I, S^0_{\J}}[[ \E^x_{S_\I, S^0_{\J}}[u_\eps(x_{\tau'})\mathbbm{1}_{\text{C}_3}|Y]]\label{cutintcond}\\
&=\int_{\R}
\left(\E^x_{S_\I, S^0_{\J}}\left[u_\eps(x_{\tau'})\mathbbm{1}_{\text{C}_3}| Y=s\right]\right)\, \mu(ds).\nonumber
\end{align}
where $\mu$ is the probability distribution of $Y$.
  Next, notice that    for any point  $z_1\in \Omega$, we have
\begin{align}
\label{eq:difference}
-\delta+b+\nu\cdot z_1\leq u_\eps(z_1)\leq \nu\cdot z_1+\delta+b.
\end{align}
This follows from the comparison Lemma \ref{lemcomp}.

 Now to illustrate,  suppose  that the original process starting at $x$ has  some realization  satisfying  $\sum_{j=0}^{\tau'} \eps_j a_j=s$. We take the corresponding paths both starting at $x$ and starting at $y$ with the same realization. Denote by $x_{\tau'}$ and $y_{\tau'}$ the end points of the paths. Recalling that one of the players  is using the cancellation strategy  the paths we need to concentrate on are of the form 
\begin{align*}
x_{\tau '}=z+(x-z)+\sum_{j\in J_1^{\tau'}} v_j+\sum_{j\in J_2^{\tau'}} w_j+\sum_{j\in J_3^{\tau'}} h_j=z+ q+\sum_{j\in J_3^{\tau'}} h_j\\
y_{\tau '}=z+(y-z)+\sum_{j\in J_1^{\tau'}} \tilde v_j+\sum_{j\in J_2^{\tau'}} \tilde w_j+\sum_{j\in J_3^{\tau'}} \tilde h_j=z+\tilde q+ \sum_{j\in J_3^{\tau'}} \tilde h_j,
\end{align*}
where $|q|,|\tilde q|\leq |x-z|+ \sum_{j=0}^{\tau'} \eps_j a_j+\eps=s+|x-z|+\eps$. Thus by (\ref{eq:difference})
\begin{align*}
u_\eps(x_{\tau'})\mathbbm{1}_{\text{C}_3}\leq \left(\nu\cdot z+ |\nu|(s+|x-z|+\eps)+\nu\cdot \sum_{j\in J_3^{\tau'}} h_j+b+\delta\right)\mathbbm{1}_{\text{C}_3}\\
u_\eps(y_{\tau'})\mathbbm{1}_{\text{C}_3}\geq \left(\nu\cdot z+ |\nu|(s+|x-z|+\eps)+\nu\cdot \sum_{j\in J_3^{\tau'}} \tilde h_j+b-\delta\right)\mathbbm{1}_{\text{C}_3}.
\end{align*}
It follows that 
\begin{align}\label{ineq1cylwal}
\E^x_{S_\I, S^0_{\J}}\left[u_\eps(x_{\tau'})\mathbbm{1}_{\text{C}_3}\,|
Y=s\right]&\leq (\nu \cdot z+\delta+b)\E^x_{S_\I, S^0_{\J}}\left[\mathbbm{1}_{\text{C}_3}|Y=s \right]
\nonumber\\
 &\quad+(|\nu|(|x-z|+s+\eps))\E^x_{S_\I, S^0_{\J}}\left[\mathbbm{1}_{\text{C}_3}| Y=s\right ]\\
&\quad+\E^x_{S_\I, S^0_{\J}}\Big[(\nu \cdot \sum_{j\in J_3^k} h_j) \mathbbm{1}_{\text{C}_3} |
Y=s\Big]\nonumber,
 \end{align}
 and
 \begin{align}\label{ineq2cylwal}
\E^y_{S^0_\I, S_{\J}}\left[u_\eps(x_{\tau'})\mathbbm{1}_{\text{C}_3}\, |
Y=s\right]\nonumber&\geq (\nu \cdot z+b-\delta) \E^y_{S^0_\I, S_{\J}}\Big[\mathbbm{1}_{\text{C}_3}|
Y=s\Big]\nonumber\\
&\quad -(|\nu|(|x-z|+s+\eps))\E^y_{S^0_\I, S_{\J}}\left[\mathbbm{1}_{\text{C}_3}|
Y=s\right] \\
&\quad+\E^y_{S^0_\I, S_{\J}}\Big[(\nu \cdot \sum_{j\in J_3^k}\tilde h_j) \mathbbm{1}_{\text{C}_3}|Y=s\Big].\nonumber
\end{align}
Observe that the last terms coincide in \eqref{ineq1cylwal} and \eqref{ineq2cylwal}.
Hence, using \eqref{cancelsym}, \eqref{cutintcond}, \eqref{ineq1cylwal}, and \eqref{ineq2cylwal},  we get that 
\begin{align}\label{sect4mainineq}
 &\E^x_{S_\I, S^0_{\J}}[u_\eps(x_{\tau'})]- \E^y_{S_\I^0, S_{\J}}[u_\eps(x_{\tau'})]\nonumber\\
 &\quad= P(\text{C}_1)(\E^x_{S_\I, S^0_{\J}}[u_\eps(x_{\tau'})|\text{C}_1]-\E^y_{S^0_\I, S_{\J}}[u_\eps(x_{\tau'})|\text{C}_1])\nonumber\\
 &\qquad+P(\text{C}_2)(\E^x_{S_\I, S^0_{\J}}[u_\eps(x_{\tau'})|\text{C}_2]-\E^y_{S^0_\I, S_{\J}}[u_\eps(x_{\tau'})|\text{C}_2])\\
 &\qquad+P(\text{C}_3)( \E^x_{S_\I, S^0_{\J}}[u_\eps(x_{\tau'})|\text{C}_3]- \E^y_{S^0_\I, S_{\J}}[u_\eps(x_{\tau'})|\text{C}_3])\nonumber\\
 &\quad\leq P(\text{C}_2)(|\nu|(|x-y|+2r+2\eps)+2\delta)\nonumber\\
 &\qquad + \int_\R (|\nu|(|x-y|+2s+2\eps)+2\delta)\mu(ds).\nonumber
 \end{align}

{\bf Step 2: Cylinder walk and modified barrier function}
  In order to estimate  $|u_\eps(x)-u_\eps(y)|$, we use the cylinder walk, that is the  following random walk in a $n+1$-dimensional cylinder $B_r(0)\times(0, r+|x-z|+\eps)$.
Rules of the walk:
\begin{itemize}
\item the token is initially at $(\zeta_0,t_0)=(0, |x-z|+\eps)$,
\item  with probability $\alpha$,  the token moves randomly from $(\zeta_j, t_j)\in B_r(0)\times(0, r+|x-z|+\eps)$  to $(\zeta_{j+1}, t_{j+1})$,
where $t_{j+1}$ is chosen according to a uniform probability distribution on $[t_j-\eps, t_j+\eps]$, and $\zeta_{j+1}:=\zeta_{j}$,
\item with probability $\beta$ the token moves randomly from $(\zeta_j, t_j)$  to  $(\zeta_{j+1}, t_{j+1})$ where $\zeta_{j+1}$ is chosen according to a uniform probability distribution on $B_\eps(\zeta_j)$, and $t_{j+1}:=t_j$.
\end{itemize}
   We denote the probability measure for the cylinder walk by $\ol P$. When computing the value for the cylinder walk, we use a new barrier function $\bar u$  compared to Section \ref{sect3} with different boundary values given below.
The function  $\bar u$ is an  explicit solution to 
    \begin{align}
    \label{eq:auxiliary-eq}
    \frac{p-2}{3}\bar u_{tt}+\Delta \bar u=0
    \end{align}
that we construct below.
The function $\bar u$ is a solution to the same equation  as the one utilized in  Step 2 in the proof of Lemma \ref{lipestrough1}  but we modify the boundary values taking into account the more precise behavior of $F$. The idea to choose  again a solution of this equation is to  be able to   use the fact that,  when taking a sequence $(x_j, t_j)$ of positions in the cylinder walk,  the sequence  $$
 M_j := \bar u(\zeta_j ,t_j ) -Cj\eps^3
 $$
  is a supermartingale.
The boundary conditions on $\bar u $ are the following:
\begin{align}
\label{eq:mod-cyl-walk-bdr-cond}
\begin{cases}
\bar u\geq (2|\nu|(r+|x-z|+\eps)+c\delta) &\text{on}\quad  B_r(0)\times\left\{r+|x-z|+\eps\right\} \quad \text{(top)},\\
 \bar u(\zeta,t)\geq2|\nu|t+2\delta &\text{on}\quad  \partial B_r(0)\times[0, r+|x-z|+\eps]\quad  \text{(sides)},\\
  \bar u\geq 0&\text{on}\quad   B_r(0)\times\left\{0\right\}\quad   \text{(bottom)},\\
\bar u(0,0)=0.&
\end{cases}
\end{align}
The choice of the side values is motivated by the following observations.
Suppose  that the original process starting at $x$ ends because of stopping condition {\bf C3} or condition {\bf C2} with some realizations $\sum_{j=0}^i \eps_j a_j$ and $\sum_{j\in J_{3}^i} h_j$. Then its associated  path in the cylinder walk  hits either  the side boundary strip of the cylinder  or the top of the cylinder at $(\zeta_{\bar \tau},t_{\bar\tau}):=(\sum_{j\in J_{3}^{\tau'}} h_j, |x-z|+\sum_{j=0}^{\tau'} \eps_j a_j+\eps)$.
 At this point $(\zeta_{\bar\tau}, t_{\bar\tau})$   we have that
\begin{align*}
\bar u(\zeta_{\bar\tau},t_{\bar\tau})&\geq 2|\nu|t_{\bar\tau}+2\delta\\
&=2|\nu|\Big(|x-z|+\sum_{j=0}^{\tau'}\eps_j a_j+\eps\Big)+2\delta.
 \end{align*}
 The case where the original process ends because of stopping condition {\bf C1} corresponds to the exiting  through the bottom strip of the cylinder in the cylinder walk, where we would like to set boundary conditions $0$. However, the explicit function that we use below might be slightly negative causing a small error.  

Next, let $\bar \tau$  be the first time the cylinder walk starting from  $(0, |x-z|+\eps)$ exits the cylinder and introduce the random variable  $\ol Y$  defined  by $$\ol Y(w)=t_{\bar\tau}-(|x-z|+\eps).$$
Define
$\text{EB}$ as the event that the cylinder walk starting from $(0,|x-z|+\eps)$ exits the cylinder through the bottom, 
$\text{ET}$ as the event that the cylinder walk starting from $(0, |x-z|+\eps)$ exits the cylinder through the top  and $\text{ES}$  as the event that the cylinder walk starting from $(0, |x-z|+\eps)$ exits the cylinder through the sides.
We have
\begin{align*}
\E[\bar u (\xi_{\bar\tau}, t_{\bar\tau})]&=\ol P(EB) \E[\bar u (\xi_{\bar \tau}, t_{\bar \tau}) |EB]+\ol P(ET)\E[\bar u (\xi_{\bar \tau},t_{\bar \tau})|  ET]\\
&\qquad+\ol P(ES) \E[\bar u (\xi_{\bar \tau},t_{\bar \tau})|  ES]\\
&=\ol P(EB) \E[\bar u (\xi_{\bar \tau}, t_{\bar \tau}) | EB]+\ol P(ET)\E[\bar u (\xi_{\bar \tau},t_{\bar \tau})| ET]\\
&\qquad+\E[ \E[\bar u (\xi_{\bar \tau},t_{\bar \tau})\mathbbm{1}_{ES}| \ol Y]]\\
&=\ol P(EB) \E[\bar u (\xi_{\bar \tau}, t_{\bar \tau}) |EB]+\ol P(ET)\E[\bar u (\xi_{\bar \tau},t_{\bar \tau})|  ET]\\
&\qquad+ \int_{\R}
\E\left[\bar u(\xi_{\bar \tau}, t_{\bar \tau})\mathbbm{1}_{ES}| t_{\bar \tau}=s+|x-z|+\eps\right]\,\bar \mu(ds)\\
&\geq \ol P(EB) \E[\bar u (\xi_{\bar \tau}, t_{\bar \tau}) |EB]+\ol P(ET)(2|\nu|(|x-z|+r+\eps)+2\delta)\\
&\qquad+ \int_{\R} (|\nu|(|x-y|+2s+2\eps)+2\delta)\,\bar \mu(ds)
\end{align*}
where $\ol \mu$ is the probability distribution of $\ol Y$. Observing that by construction the involved probabilities are the same as in (\ref{sect4mainineq}), we may combine this with (\ref{sect4mainineq}) and obtain 
\begin{align*}
|u_\eps(x)-u_\eps(y)| &\leq\underset{S_\I, S_\J}{\sup} \left|\E^x _{S_\I, S_{\J}^0}[u_\eps(x_{\tau'})]-\E^y_{S_\I^0, S_{\J}}[u_\eps(x_{\tau'})]\right|\\
&\leq \E[\bar u(\zeta_{\bar\tau }, t_{\bar\tau })] - \underset{B_r(0)}{\inf} \bar u(\zeta, -\eps)
\end{align*}
The term $- \inf_{B_r(0)}\bar u(\zeta, -\eps)$ on the last line arises from the fact that our explicit function constructed below can be slightly negative in the bottom strip of the cylinder. 

{\bf Step 3: construction of the barrier function $\bar u$} In order to construct  an explicit solution $\bar u$  as mentioned above, we  define the following domain (see Figure  \ref{cyl2}).
The center of the bottom is at $(0,0)$ and otherwise the bottom is a part of an ellipsoid $E_1$ centered at $\left(0, -\frac{\sqrt{p-2}}{\sqrt 3}R\right)$,
$$E_1:=\left\{ (\zeta,t)\in \Rn\times \R\,:\,  |\zeta|^2+\left(\frac{\sqrt{3}t}{\sqrt{p-2}}+R\right)^2=R^2\right\}$$
with
$2r\leq R$. 
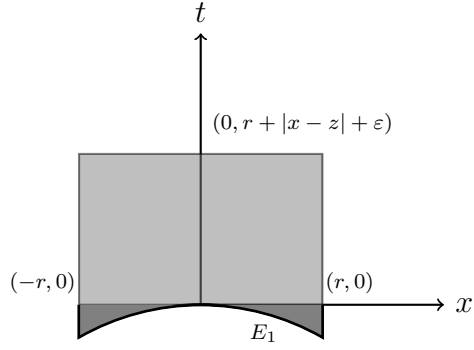
\begin{figure}[h]
\begin{center}
\begin{tikzpicture}[thick, scale=.8]
\draw [<->,thick] (0,4.5) node (yaxis) [above] {$t$}|- (4,0) node (xaxis) [right] {$x$};
\draw[fill=gray, opacity=0.5] (-2,0)--(-2,2.5)--(2,2.5)--(2,0)--(-2,0);
\draw[very thick](0,0) arc (90:60:4cm)--(2,0);
\draw[very thick](0,0) arc (90:120:4cm)--(-2,0); 
\draw arc (90:60:4cm) node[midway,below]{\tiny {$E_1$}};
\draw[fill=gray](0,0) arc (90:60:4cm)--(2,0);
\draw[fill=gray](0,0) arc (90:120:4cm)--(-2,0);
\draw (2.45,0) node[above]{\tiny{$(r, 0)$}};
\draw (-2.6,0) node[above]{\tiny{$(-r, 0)$}};
\draw (0,3) node[right]{\tiny{$(0, r+|x-z|+\eps)$}};
\end{tikzpicture}
\end{center}
\caption{One dimensional  illustration related to Step 2 and 3 in the proof of Theorem \ref{thm:lip-est}}
\label{cyl2}
\end{figure}
Let $C>2$, and define the function
$$\bar u(\zeta,t):= 2|\nu|t +C\delta\left[\frac{R^{1-n}-\left[|\zeta|^2+(\frac{\sqrt{3}t}{\sqrt{p-2}}+R)^2\right]^{(1-n)/2}}{r^{1-n}-R^{1-n}}\right]$$
that satisfies (\ref{eq:auxiliary-eq}) and (\ref{eq:mod-cyl-walk-bdr-cond}). Observe that this function defines a solution also in the $\eps$-strip outside a domain.

{\bf Step 4: Estimate of the value function $\bar u$ on the cylinder walk  giving the estimate of $|u_\eps(x)-u_\eps(y)|$}
It follows from the  Taylor expansion that
$$\bar u(\zeta ,t)=\vint_{B_\eps(\zeta)} \bar u(z,t)\,dz-\dfrac{\Delta \bar u(z\eta,t)\eps^2}{2(n+2)}+ C\eps^3,$$
$$\bar u(\zeta,t)=\frac{1}{2\eps}\int_{t-\eps}^{t+\eps} \bar u(\zeta,h) \, dh-\dfrac{\eps^2 \bar u_{tt}(x,t)}{6}+ C\eps^3.$$
Hence,  using that  $$\frac{\beta}{2(n+2)}\Delta \bar u+\frac{\alpha}{6} \bar u_{tt}=\frac{1}{2(p+n)}\left[\Delta \bar u+\frac{p-2}{3}\bar u_{tt}\right]=0,$$
we have 
\begin{equation}\label{Marki}
\bar u(\zeta,t)=\beta \vint_{B_\eps(\zeta)} \bar u(z,t)\,dz+ \frac{\alpha}{2\eps}\int_{t-\eps}^{t+\eps} \bar u(\zeta,h) \, dh+ C\eps^3,
\end{equation}
where $C$ depends on $n$ and the third derivatives of $\bar u$.

Consider the sequence of random variables $\bar u(\zeta_j , t_j )$,
$j =0, 1, 2, . . . $, where  $(\zeta_j ,t_j )_{j\in \N}$ are the positions in the cylinder walk.
From \eqref{Marki}, we have that
 $$
 M_j := \bar u(\zeta_j ,t_j ) -Cj\eps^3
 $$
  is a supermartingale. Then, applying the optional
stopping, using the stopping time $\bar\tau$
that corresponds to exit from the domain $B_r(0)\times[0, r+|x-z|]$, we
get  that 
$$\E[\bar u(\zeta_{\bar\tau},t_{\bar\tau})
- C\eps^3\bar\tau
] \leq M_0 = \bar u(0, t_0)\leq |\partial_t \bar u(0, c)| t_0,$$
for some $c\in (0, t_0)$, $t_0=|x-y|/2$. Rearranging, we get 
\begin{align}
\label{winni2}
 \E[\bar u(\zeta_{\bar\tau }, t_{\bar\tau })]&\leq  C\eps^3\E[\bar\tau ]+ t_0|\partial_t \bar u(0, c)|.
\end{align}
It remains to estimate the terms on the right hand side.

 First we estimate  the $t$-derivative of $\bar u$. Observe
$$\partial_t \bar u(\zeta,t)=2|\nu|+c\delta(n-1) \frac{\sqrt{3}}{\sqrt{p-2}} \frac{\left[|\zeta|^2+(\frac{\sqrt 3 t}{\sqrt{p-2}}+R)^2\right]^{(-n-1)/2}(\frac{\sqrt 3t}{\sqrt{p-2}}+R)}{r^{1-n}-R^{1-n}},$$
so that  for $(\zeta,t)\in B_r(0)\times[0, r+|x-z|]$, we have  
\begin{align*}|
\partial_t \bar u(\zeta,t)|&\leq2|\nu|+ c(n-1)\delta\frac{\sqrt 3}{\sqrt{p-2}}\frac{\left[|\zeta|^2+(\frac{\sqrt 3 t}{\sqrt{p-2}}+R)^2\right]^{-n/2}}{r^{1-n}-R^{1-n}}\\
&\leq2|\nu|+c\delta (n-1)\frac{\sqrt 3}{R^{1-n}\sqrt{p-2}}\frac{R^{-n}}{\left[\left(\frac{r}{R}\right)^{1-n}-1\right]}\\
&\leq 2|\nu|  +C\delta
\end{align*}
where in the last inequality we estimated $\left[ (r/R)^{1-n}-1\right]^{-1}\leq 1$ by using that $R\geq 2 r$.
In order to estimate the error, notice that, for $(\zeta,t)\in B_{2r}\times[-\eps, r+|x-z|+\eps]$, we have
\begin{align*}
|D_{\zeta,t}^3 \bar u(\zeta,t)|&\leq \delta C(n,p) \frac{\left[|\zeta|^2+(\frac{\sqrt 3t}{\sqrt{p-2}}+R)^2\right]^{-(n+2)/2}}{r^{1-n}-R^{1-n}}\\
&\leq \delta C(n, p)\frac{R^{-2-n}}{r^{1-n}-R^{1-n}}\\
&\leq C(p,n)\delta
\end{align*}
by estimating that for $\eps$ small $\left[|\zeta|^2+(\frac{\sqrt 3t}{\sqrt{p-2}}+R)^2\right]^{-(n+2)/2}\leq C R^
{-(n+2)/2}$.
Next using this estimate and proceeding in a similar way as in \cite{luirops13},  we estimate  $\E[\bar\tau]$ by
$$\E[\bar\tau ]\leq C(n)\eps^{-2}.$$
It follows that we can estimate  the right hand side of (\ref{winni2}) by
\begin{align*}
C\eps^3\E[\bar\tau ]\le C(p,n)\delta\frac{\eps^3}{\eps^2}\leq C\delta \eps.
\end{align*}
Finally, combining the estimates of Step 4 with \eqref{winni2}, we have  
 \begin{align}
|u_\eps(x)-u_\eps(y)| &\leq \E[\bar u(\zeta_{\bar\tau }, t_{\bar\tau })] - \underset{B_r(0)}{\inf} \bar u(\zeta, -\eps)\nonumber\\
&\le (2|\nu|+ C\delta)(|x-z|+\eps)+ C\delta \eps- \underset{B_r(0)}{\inf} \bar u(\zeta, -\eps)\\
&\leq  (|\nu|+ C\delta)|x-y|+(5|\nu|+ C\delta) \eps.\nonumber
\end{align}
\end{proof}
\begin{remark} In dimension $n=1$ we would use the function $$\bar u(\zeta,t):=c\delta\left( \frac{\log\left(\sqrt{|\zeta|^2+(\frac{\sqrt 3t}{\sqrt{p-2}}+R)^2}\right)-\log(R)}{\log(r)-\log(R)}\right)+2|\nu|t.$$
It can be shown that the estimates  $|\partial_t \bar u|\leq (3|\nu|+C\delta)$ and $|D^3 \bar u|\leq C\delta$ still  hold.
\end{remark}

\appendix
\renewcommand{\theequation}{\thesection.\arabic{equation}}

\section{Existence and uniqueness of functions satisfying the  DPP}\label{appendA}
In this section we prove the existence and uniqueness of the value of the game for the random step size TWN. 
The proof is an easy adaptation of the arguments of \cite{luirops14}. In the case of the obstacle problem, the existence is considered in \cite{lewickam17} and in the case  $p=\infty$  in \cite{lius15}.
\begin{lemma}[Existence for DPP]\label{existdppn}
There exists a bounded Borel function $u_\eps$ satisfying the DPP 
\begin{align*}
u_{\eps}(x)=\frac{\alpha}{2\eps}\int_{0}^{\eps} (\sup_{B_{t}(x)} u_{\eps}+\inf_{B_{t}(x)} u_{\eps}  )\ud t +\beta \vint_{B_\eps(x)} u_\eps\, dz,
\end{align*} 
for $x\in \Omega$ and $u_\eps=F$ in $\Gamma_\eps$.
\end{lemma}
\begin{proof}
We can check that for a  Borel function $u$
the functions
$$x\mapsto\frac{1}{\eps}\int_0^\eps \underset{B_{t}(x)}{\sup} \, u(y)\, dt, \qquad \qquad x\mapsto\frac{1}{\eps}\int_0^\eps \underset{B_{t}(x)}{\inf} \, u(y)\, dt$$ are also Borel functions.
Now consider the following  iteration process $u_{j+1}:= T(u_j)$
where 
\begin{equation*}
T(u)(x)=
\begin{cases}
\frac{\alpha}{2\eps}\int_{0}^{\eps} (\sup_{B_{t}(x)} u+\inf_{B_{t}(x)} u )\ud t +\beta \vint_{B_\eps(x)} u(z)\, dz\qquad  &\text{for}\, x\in \Omega\\
F(x)\qquad&\text{for}\, x\in \Gamma_\eps,
\end{cases}
\end{equation*}
and the first function is 
\begin{equation*}
u_0(x)=
\begin{cases}
\underset{y\in \Gamma_\eps}{\inf} F(y)\qquad  &\text{for}\, x\in \Omega,\\
F(x)\qquad&\text{for}\, x\in \Gamma_\eps.
\end{cases}
\end{equation*}
 The sequence $u_j$ is increasing and  bounded from  above
by $$\underset{y\in \Gamma_\eps}{\sup} F(y).$$ It follows that  $u_j$ converges to a function $ u_\eps$ when $j\to \infty$.
Proceeding by contradiction, we can show that the convergence is uniform. Indeed, if this is not true, then,
$$A=\lim_{j\to\infty} \underset{x\in \Omega_\eps}{\sup}\, (u_j(x)-u_\eps(x))>0.$$
For any $\eta>0$ we may find $x_0\in \Omega$ such that  for $l>k$ large enough, it holds
$$u_{l+1}(x_0)-u_{k+1}(x_0)\geq A-2\eta.$$
Moreover,  using  the  the dominated convergence theorem, we may also assume that
$$\underset{x\in \Omega}{\sup} \vint_{B_\eps(x)}
u_\eps(y)-u_k(y)\, dy\leq \eta.$$
It follows that 
\begin{align*}
A-2\eta&\leq u_{l+1}(x_0)-u_{k+1}(x_0)\\
&=\frac{\alpha}{2\eps}\int_{0}^{\eps} (\sup_{B_t(x_0)} u_{l}+\inf_{B_t(x_0)} u_{l}  )\ud t +\beta \vint_{B_\eps(x_0)} u_l(z)\, dz\\
&\quad-\frac{\alpha}{2\eps}\int_{0}^{\eps} (\sup_{B_t(x_0)} u_{k}+\inf_{B_t(x_0)} u_{k}  )\ud t +\beta \vint_{B_\eps(x_0)} u_k(z)\, dz\\
&\leq \alpha\underset{B_\eps(x_0)}{\sup}(u_l-u_k)+\beta\vint_{B_\eps(x_0)} (u_l-u_k)(z)\, dz\\
&\leq \alpha\underset{B_\eps(x_0)}{\sup}(u_\eps-u_k)+\beta\vint_{B_\eps(x_0)} (u_\eps-u_k)(z)\, dz\\
&\leq \alpha(A+\eta)+\eta.
\end{align*}
Here we used that 
$$
\underset{B_t(x_0)}{\sup}(u_l-u_k)\leq \underset{B_\eps(x_0)}{\sup}(u_l-u_k)\quad \text{for}\quad t\in [0, \eps].$$
We get that $(1-\alpha)A\leq (\alpha+3)\eta$ and we end up with a contradiction if we choose $0<\eta<\frac{(1-\alpha)A}{2(\alpha+3)}$.
The uniform convergence of $u_j$ to $u_\eps$  implies that we can pass to the limit in the DPP functional and hence  the limit $u_\eps$ obviously satisfies the DPP and
it has the right boundary values by construction.
\end{proof}
The uniqueness of the  function $u_\eps$ satisfying the DPP \eqref{dpp1}  and having boundary values $F$ is a consequence of  the following lemma.
\begin{lemma}[Comparison]
Let $u_\eps$ and $\ol u$ be  bounded functions satisfying the DPP  \eqref{dpp1} in $\Omega$ and  $\ol u\ge u_\eps $ on $\Gamma_\eps$.
Then it holds
$$\bar u\geq u_\eps\quad\text{in}\quad \Omega_\eps.$$
\end{lemma}

\begin{proof}
 We argue by contradiction.
Assume that $u_\eps(y) > \bar u(y)$ for some $y\in\Omega$. Since $u_\eps -\bar u$ is bounded, we
have $ \underset{\Omega}{\sup}\, (u_\eps-\bar u) =: M > 0$.
Using the  DPP \eqref{dpp1}, we have
\begin{align}\label{nonemptyset}
u_\eps(x)-\bar u(x)&=\frac{\alpha}{2\eps}\int_{0}^{\eps} (\sup_{B_{t}(x)} u_{\eps}-\sup_{B_{t}(x)} \bar u )\,dt
-\frac{\alpha}{2\eps}\int_{0}^{\eps} (\inf_{B_{t}(x)} u_{\eps}-\inf_{B_{t}(x)} \bar u )\ud t\nonumber\\
&\quad +\beta\vint_{B_\eps(x)} u_\eps(z)-\bar u(z)\, dz\nonumber\\
&\leq \alpha M+\beta \vint_{B_\eps(x)} u_\eps(z)-\bar u(z)\, dz.
\end{align}
The inequality \eqref{nonemptyset} and the  absolute continuity of the integral imply that the set 
$$G := \left\{x\in \Omega_\eps : u_\eps(x) -\bar u(x) = M\right\}$$ 
 is   non-empty and also satisfies $G\subset \Omega$ by using the boundary data assumption.
We  deduce that, if $\zeta\in G$, then $u_\eps - \bar u = M$ almost everywhere in a ball $B_\eps (\zeta)$. By continuing, this contradicts the assumption that  $G\subset\Omega$. 
\end{proof}
The previous lemma also holds if we reverse the inequalities. Thus it implies that the function $u_\eps$ satisfying the DPP  \eqref{dpp1} with $u_\eps=F$ on $\Gamma_\eps$ is unique. Now we are ready to show that the game has a value.

\begin{lemma}\label{gamevalueint}
Let $u_\eps$ be the unique bounded function satisfying the DPP  \eqref{dpp1} with $u_\eps=F$ on $\Gamma_\eps$. Let $u^\eps_{\text{I}}$ be  the value of the game  for  Player I and $u^\eps_{\text{II}}$ be the value function of the game for Player II.
Then $u^\eps_\J= u_\eps= u_\I^\eps$.
\end{lemma}
\begin{proof}
Since we always have $u^\eps_\I \leq u^\eps_\J$, in order to show that $u_\eps=u_{\I}^\eps=u_{\J}^\eps$, it is enough to prove that $u^\eps_\J\leq u_\eps\leq u_\I^\eps$.
We will only show that $u^\eps_\J\leq u_\eps$ since  the proof of $u_\I^\eps\geq u_\eps$ is analogous.

Fix a point $x\in \Omega$, a starting point for a game.  Player I plays with any strategy and Player II plays with the following strategy $S_\J^0$. From a point $x_{k-1}\in \Omega$ taken that the radius $t$ has been selected,   Player II steps to a point $x_k\in B_{t}(x_{k-1})$ such that
$$ u_\eps(x_k) \leq \inf_{ B_{t}(x_{k-1})} u_\eps+\eta 2^{-k}$$
for some fixed $\eta>0$.
In order to ensure that this kind of
strategies are Borel,  we can  adapt  the arguments used in \cite{luirops14}.
Then  we have
\begin{align*}
\E^{x_0}_{S_\I, S^0_\J}\left[u_\eps(x_k)+\eta 2^{-k}|\F_{k-1}\right]\leq &\frac{\alpha}{2\eps} \int_0^\eps \left(\inf_{B_t(x_{k-1})}  u_\eps+\sup_{B_t(x_{k-1})}  u_\eps \right) \, dt  \\
&\quad+\beta\vint_{B_\eps(x_{k-1})}u_\eps(z)\, dz+\frac{3}{2}\eta 2^{-k} \\
&\leq  u_\eps(x_{k-1})+\eta 2^{-(k-1)}.
\end{align*}
It follows that the process  $M_k:= u_\eps(x_k)+\eta 2^{-k}$ is a supermartingale when using the strategies $S_\I$ and $S_\J^0$.  It follows that,
\begin{align*}
u_{\J}^\eps (x_0)&=\inf_{S_\J}\sup_{S_I} \E^{x_0}_{S_\I, S_\J}\left[F(x_\tau)\right]\\
&\leq \sup_{S_\I} \E^{x_0}_{S_\I, S^0_\J}\left[F(x_\tau)\right]\\
&\leq \sup_{S_\I}\E^{x_0}_{S_\I, S^0_\J}\left[u_\eps(x_\tau)+\eta 2^{-\tau}\right]\\
&\leq \sup_{S_\I}\E^{x_0}_{S_\I, S^0_\J}\left[M_0\right]\\
&\leq u_\eps(x_0)+\eta.
\end{align*}
Since  $\eta>0$ was arbitrarily chosen, we get that $u^\eps_\J\leq u_\eps$. A similar argument  where Player II chooses any strategy and Player I steps to a point almost maximizing $u_\eps$ gives  that 
$u_{\I}^\eps\geq u_\eps$ in $\Omega_\eps$.
\end{proof}
\section{Relation to $p$-harmonic functions}\label{appendixB}

We establish the convergence of the value functions to $p$-harmonic functions i.e.\ the details skipped in the  proof of Theorem~\ref{thm:conv-of-gradients}.
The proof following \cite{manfredipr12} contains two parts: the compactness estimates which allow to prove that $u_\eps\to v$ and the identification of the limit which allows to state that $v=u$.
\begin{lemma}[Relation to $p$-harmonic functions]\label{identylimp}
Let $F\in C(\Gamma_\eps)$, $2<p<\infty$ and let $u_\eps$ be the value function of the  random step size TWN with  boundary values  $F$.  Assume that $\Om$ satisfies a uniform exterior sphere condition.  
Let $u$ be the unique $p$-harmonic function in $\Omega$ with $u=F$ on $\partial\Omega$. Then 
\begin{align*}
u_{\eps}\to u.
\end{align*}
\end{lemma}

The functions $u_\eps$  are locally Lipschitz but they may be discontinuous near the boundary. 
One can modify the game near the boundary to get continuous functions, but also without that, we can show the convergence of the functions $u_\eps$ when $\eps\to 0$ by using for example 
the following  variant of the Arzel\'a-Ascoli theorem which is Lemma 4.2 in  \cite{manfredipr12}.

\begin{lemma}\label{asymeq}
Assume that  $\left(u_\eps\right)$ is a
uniformly bounded set of functions and that for any  given $\eta>0$, there are constants $r_0$ and $\eps_0$ such that for every
$\eps < \eps_0 $ and any  $x, y\in \overline\Omega$ with
$|x -y|\leq r_0$
it holds $$|u_\eps(x)-u_\eps(y)|\leq \eta.$$
Then there exist a subsequence that we still denote by $u_\eps$ and a uniformly continuous function $u$ such that $u_\eps \to u$ uniformly in $\overline \Omega$.
\end{lemma}

Using barrier arguments and the local Lipschitz estimate from  Theorem \ref{thm:main}, we can  show that $u_\eps$ satisfy the conditions of Lemma \ref{asymeq}. 
\begin{lemma}\label{equicomp}
Let $u_\eps$ be as in Lemma~\ref{identylimp}. Then $(u_\eps)$ satisfy the conditions of Lemma \ref{asymeq}.
\end{lemma}
\begin{proof}
If $x$ and $y$ are in $\Gamma_\eps$, the result follows from the continuity of $F$. 

Next, let us show that this holds at the vicinity of  the boundary by using a barrier argument.
Let $y\in \partial \Omega$, $x\in \Omega$, and $\eta>0$. We would like to show that  for some $r_1>0$ and $|x-y|\leq r_1$ we have
$$|u_\eps(x)-F(y)|\leq \eta.$$
Since $\Omega$ satisfies the exterior sphere condition, we have $y\in \partial B_r(z)$ for some $B_r(z)\subset \R^n\setminus \Omega$. Take $R>r$ such that $\Omega\subset B_R(z)$.
We start the game from $x$  and choose a strategy $S_{\J}^0$ for Player II where he pulls towards $z$. Player I plays with a strategy $S_\I$.

 In this case we have
\begin{align*}
\E^x_{S_\I, S_\J^0}[|x_k-z|^2\,|]&\leq  \frac{\alpha}{2\eps}\int_0^\eps ((|x_{k-1}-z|-t)^2
+( |x_{k-1}-z|+t)^2)\, dt\\
&\hspace{10 em}+\beta \vint_{B_\eps(x_{k-1})} |h-z|^2\, dh\\
&\leq |x_{k-1}-z|^2+C(n)\eps^2.
\end{align*}

Hence $M_k:=|x_k-z|^2-C(n)\eps^2 k$ is a supermartingale.
It follows that 
$$\E^x_{S_\I, S_\J^0}[|x_\tau-z|^2]\leq  |x-z|^2+C\eps^2 E^x_{S_\I, S_{\J}^0}[\tau].$$
Now we have to estimate $ \E^x_{S_I, S_{\J}^0}[\tau]$. Assume that $\Omega\subset B_R(z)$.
We show that 
\begin{equation}
\E^x_{S_\I, S_{\J}^0}[\tau]\leq \dfrac{C(R/r)\dist(\partial B_r(z), x)+o(1)}{\eps^2}
\end{equation}
where $o(1)\to 0$ when $\eps\to 0$. 
Indeed, consider a solution $\bar v$ to 
\begin{equation*}
\begin{cases}
-\Delta \bar v=2(n+2)&\text{in}\, B_{R+\eps}(z)\setminus B_r(z)\\
\bar v=0&\text{on}\, \partial B_r(z)\\
\frac{\partial \bar v}{\partial\nu}=0&\text{on}\, \partial B_{R+\eps}(z)
\end{cases}
\end{equation*}
and extend it as a solution slightly inside $B_r$.
The function $\bar v$ is concave and satisfies
\begin{equation}\label{concabarv}\bar v(x)=\vint_{B_\eps(x)} \bar v(y)\, dy+\eps^2
\end{equation}
in the annulus and near  $\partial B_r(z)$.
The concavity of $\bar v$  implies  that 
$$ 
\dfrac{\inf_{B_t(x)} \bar v+\sup_{B_t(x)}\bar v}{2}\leq \bar v(x) 
$$
 and together with  \eqref{concabarv} we get that $\bar v(x_k)+k\eps^2$ is a supermartingale (by using a pulling towards $z$ strategy in the whole annulus).
Defining the stopping time $\bar \tau$  as 
$$\bar \tau:=\inf \left\{k\,:\, x_k \in  \ol B_r(z)\right\},$$
it follows  that 
$$\eps^2\E^x_{S_\I, S_{\J}^0}[\bar \tau]\leq \ol v(x)-\E^x_{S_\I, S_{\J}^0}[\ol v(x_{\bar\tau})]\leq (R/r)\dist(\partial B_r(z), x)+o(1),$$
where the process is defined through a reflection at the outer boundary, see \cite[Lemma 4.5]{manfredipr12}.
Since $\tau\leq \bar \tau$, we get the desired estimate.\\
The  triangle inequality and the uniform continuity of the boundary function together with the previous estimate  give the desired result for $x\in \Omega$ and $y\in \Gamma_\eps$:
there exist $r_0> 0$ and
$\eps_1 > 0$ such that  if $|y -x|<r_0$, we have
$$|u_\eps(x)-u_\eps(y)|\leq \eta/2.$$
The triangle inequality also gives the desired result for  points $x, y\in \Omega$ and satisfying $\dist(\left\{x, y\right\}, \Gamma_\eps)\leq r_0/2$.
 Next, when $ \dist({x, y}, \Gamma_\eps) \geq \frac{r_0}{2}$, we use the  local Lipschitz continuity of $u_\eps$ to get the desired result.
\end{proof}
\noindent{\bf Identifying the limit}.
Next, we prove that the limit function $u$ is a $p$-harmonic function. The proof is similar to \cite{manfredipr10}. Observe that from \cite{juutinenlm01}(usual $p$-Laplacian) and \cite{kawohlmp12} (normalized $p$-Laplacian), we can restrict the class of test functions $\varphi$ to those with  non vanishing gradient at the contact points.

Let $\varphi$ be a smooth test function and suppose that $\varphi$ touches $u$ from below at $x\in \Om$ and that $D\varphi(x)\neq 0$.  From the uniform convergence of $u_\eps$, we get that there exists  a sequence $x_\eps$  that converges to $x$ and such that  
\begin{equation}\label{jyva}
u_\eps(x_\eps)-\varphi(x_\eps)\leq u_\eps(x)-\varphi(x).
\end{equation}
Without loss of generality, we can assume that $\varphi(x_\eps)=u_\eps(x_\eps)$.
Using that $u_\eps$ satisfies the DPP, we get  (plugging the inequality \eqref{jyva} into the DPP) that 
\begin{equation}\label{dpper}
\frac{\alpha}{2\eps}\left\{ \int_0^\eps \left(\inf_{B_t(x_\eps)} \varphi+\sup_{B_t(x_\eps)} \varphi\right) \, dt\right\}+ \beta\vint_{B_\eps(x_\eps)} \varphi(z)\, dz\leq u_\eps(x_\eps)=\varphi(x_\eps).
\end{equation}

Denote by  $\bar x^t_\eps$
a point in which $\varphi$ attains its minimum
over a ball $\ol B_t(x_\eps)$.
Evaluating the Taylor expansion at  $y=\bar x^t_\eps$ and then at the opposite point $y=2x-\bar x^t_\eps$, we have
\begin{align*}
\varphi(\bar x^t_\eps)=\varphi(x_\eps)+D\varphi(x_\eps)\cdot (\bar x^t_\eps-x_\eps)+\frac{1}{2} D^2\varphi(x_\eps)( \bar x^t_\eps-x_\eps)\cdot  (\bar x^t_\eps-x_\eps)+o(t^2).
\end{align*}
and
\begin{align*}
\varphi(2x_\eps-\bar x^t_\eps)=\varphi(x_\eps)+D\varphi(x_\eps)\cdot (x_\eps-\bar x^t_\eps)+\frac{1}{2} D^2\varphi(x_\eps)( \bar x^t_\eps-x_\eps)\cdot(\bar x^t_\eps-x_\eps)+o(t^2).
\end{align*}
Hence adding these two expressions, we get that 
$$(\varphi(\bar x^t_\eps)+\varphi(2 x_\eps-\bar x^t_\eps))=2\varphi(x_\eps)+ D^2\varphi(x_\eps)( \bar x^t_\eps-x_\eps)\cdot (\bar x^t_\eps-x_\eps)+o(t^2).$$
Moreover  using that $\bar x^t_\eps$ is the minimum, it holds that
\begin{align*}
\frac{1}{2\eps}\left\{ \int_0^\eps \left(\inf_{B_t(x_\eps)} \varphi+\sup_{B_t(x_\eps)} \varphi\right) \, dt\right\}\geq\frac{1}{2\eps}\int_0^\eps (\varphi(\bar x^t_\eps)+\varphi(2x_\eps-\bar x^t_\eps))\, dt.
\end{align*}
Consequently
\begin{align}\label{minl1}
\frac{\alpha}{2\eps}&\left\{ \int_0^\eps \left(\inf_{B_t(x_\eps)} \varphi+\sup_{B_t(x_\eps)} \varphi\right) \, dt\right\}\\
&\geq\frac{\alpha}{2\eps}\int_0^\eps \big(2\varphi(x_\eps)+D^2\varphi(x_\eps)( \bar x^t_\eps-x_\eps)\cdot  (\bar x^t_\eps-x_\eps)\big)\, dt+ o(\eps^2).\nonumber
\end{align}
We also have 
\begin{align}\label{minl2}
\beta\vint_{B_\eps(x_\eps)} \varphi(z)\, dz=\beta\varphi(x_\eps)+\frac{\beta\eps^2}{2(n+2)}\Delta \varphi(x_\eps)+ O(\eps^3).
\end{align}
 Adding \eqref{minl1} and \eqref{minl2} and subtracting $\varphi(x_\eps)$, we get that 
 \begin{align*}
H:&= \frac{\alpha}{2\eps} \int_0^\eps \left(\inf_{B_t(x_\eps)} \varphi+\sup_{B_t(x_\eps)} \varphi\right) \, dt+ \beta\vint_{B_\eps(x_\eps)} \varphi(z)\, dz-\varphi(x_\eps)\\
 &\geq\frac{\alpha}{2\eps}\int_0^\eps (D^2\varphi(x_\eps), (x_\eps-\bar x^t_\eps) (x_\eps-\bar x^t_\eps)\, dt+\frac{\beta\eps^2}{2(n+2)}\Delta \varphi(x_\eps)+o(\eps^2). 
 \end{align*}
 Combining the above inequality with \eqref{dpper}, we get that 
 \begin{align*}
o(\eps^2)+ \frac{\alpha}{2\eps}\int_0^\eps (D^2\varphi(x_\eps), (x_\eps-\bar x^t_\eps) (x_\eps-\bar x^t_\eps)\, dt+\frac{\beta\eps^2}{2(n+2)}\Delta \varphi(x_\eps)\leq 0.
 \end{align*}
 Remembering that $\alpha=\frac{(p-2)\beta}{n+2}$, we get that 
  \begin{align*}
o(\eps^2)+\frac{\beta\eps^2}{2(n+2)}\left[ \frac{(p-2)}{\eps}\int_0^\eps \Big(D^2\varphi(x_\eps)\frac{(x_\eps-\bar x^t_\eps)}{\eps}\cdot \frac{ (x_\eps-\bar x^t_\eps)}{\eps}\Big)\, dt+\Delta \varphi(x_\eps)\right]\leq 0.
 \end{align*}
 Since $D\varphi(x)\neq 0$, the regularity of $\varphi$ implies that $D\varphi(y)\neq 0$ in a neighborhood of  $x$ and hence for $\eps$ small
 $\bar x_\eps^t\in \partial B(x_\eps, t)$.
 It follows that 
 $$\frac{1}{\eps}\int_0^\eps\Big(D^2\varphi(x_\eps)\frac{(x_\eps-\bar x^t_\eps)}{\eps}\cdot \frac{ (x_\eps-\bar x^t_\eps)}{\eps}\Big)\, dt\to \frac{D^2\varphi(x) D\varphi(x)\cdot  D\varphi(x)}{|D\varphi(x)|^2}$$
 as $\eps\to 0$.
 It follows that $u$ is a viscosity supersolution. We can similarly show  that $u$ is a subsolution by using a reverse inequality to \eqref{minl1}.
 Finally the uniqueness of $p$-harmonic functions implies the convergence of the whole sequence.

\section{Random walk with varying step size}

\label{sec:stochastic-proof-random}
Consider a symmetric random walk with varying step size. From $t_0$ we go with probability 1/2 to $t_1$ where $t_1$ is randomly chosen in $ [t_0, t_0+\eps]$ and with probability 1/2 we move to $t_1$  where $t_1$ is randomly chosen in $ [t_0-\eps, t_0]$.
We denote by $t_0, t_1, \ldots $ the positions  of this walk on the real axis. The random walk is stopped upon reaching $(0, 1)^c$ and we denote by $\tau$ the associated stopping time.
\begin{lemma}[Random walk with varying step size]\label{tcompo} Let $\eps$ small enough, then
$$\P(t_{\tau }\leq 0)\geq 1-(t_0+\eps)$$
 and 
 $$\E[\tau]\leq \frac{t_0+4\eps}{\eps^2}.$$
\end{lemma}

\begin{proof}
We use that $t_j$ is a martingale and that the optional stopping theorem implies that 
$$t_0=\E[t_{\tau}]\geq -\eps \P(t_{\tau }\leq 0)+ (1-\P(t_{\tau }\leq 0))\cdot 1.$$
This gives 
$$\P(t_{\tau }\leq 0)\geq \frac{1-t_0}{1+\eps}\geq 1-t_0-\eps.$$
Observe that 
$$\left(t_j+\frac{1}{\eps}\int_0^\eps s\,ds\right)^2+\left(t_j-\frac{1}{\eps}\int_0^\eps s\,ds\right)^2= 2t_j^2+\frac{\eps^2}{2}.$$
It follows that $t_j^2-\frac{j}{4}\eps^2$ is a martingale.
The optional stopping theorem implies that 
\begin{align*} 0\leq t_0^2=\E[t_{\tau}^2]-\eps^2\E[\tau]&\leq \eps^2\P(t_{\tau }\leq 0)+(1-\P(t_{\tau }\leq 0))(1+\eps)^2-\eps^2\E[\tau]\\
&\leq (1+\eps)^2-\P(t_{\tau }\leq 0)-\eps^2 \E[\tau].
\end{align*}
Consequently 
\[\E[\tau]\leq \frac{(1+\eps)^2-\P(t_{\tau }\leq 0)}{\eps^2}\leq\frac{t_0+4\eps}{\eps^2}.\qedhere\]
\end{proof}

\noindent \textbf{Acknowledgments.} Amal Attouchi was supported by the Academy of Finland. We thank the referees  whose comments considerably improved the manuscript.


\def\cprime{$'$} \def\cprime{$'$} \def\cprime{$'$}

\end{document}